\documentclass[10pt,letterpaper]{article}
\makeatletter
\usepackage{mathrsfs}
\usepackage{amscd}
\usepackage{amsmath}
\usepackage{latexsym}
\usepackage{amsfonts}
\usepackage{amssymb}
\usepackage{amsthm}
\usepackage{graphicx}
\usepackage[dvipdfm,
            pdfstartview=FitH,
            CJKbookmarks=true,
            bookmarksnumbered=true,
            bookmarksopen=true,
            colorlinks=true,
            pdfborder=001,
            citecolor=magenta,
            linkcolor=blue,
           ] {hyperref}

\newtheorem{theorem}{\textbf Theorem}[section]
\newtheorem{lemma}{\textbf Lemma}[section]
\newtheorem{remark}{\textbf Remark}[section]

\newtheorem{proposition}{\textbf Proposition}[section]

\newcommand{\be}{\begin{eqnarray}}
\newcommand{\ee}{\end{eqnarray}}

\newcommand{\bes}{\begin{eqnarray*}}
\newcommand{\ees}{\end{eqnarray*}}

\begin{document}
\begin{titlepage}
\title{\bf  Global existence of  strong solutions to the multi-dimensional inhomogeneous incompressible MHD equations}
\author{ Baoquan Yuan\thanks{Corresponding Author: B. Yuan}\ and
Xueli Ke
       \\ School of Mathematics and Information Science,
       \\ Henan Polytechnic University,  Henan,  454000,  China.\\
        (bqyuan@hpu.edu.cn, kexueli123@126.com)
          }
\date{}
\end{titlepage}
\maketitle
\begin{abstract}
This paper is concerned with the Cauchy problem of the
multi-dimensional incompressible magnetohydrodynamic equations with
inhomogeneous density and fractional dissipation. It is shown that
when $\alpha+\beta=1+\frac{n}{2}$ satisfying $1\leq \beta\leq \alpha\leq\min \{\frac{3\beta}{2},\frac{n}{2},1+\frac{n}{4}\}$ and $\frac{n}{4}<\alpha$ for  $n\geq3$ , then the inhomogeneous incompressible MHD equations has a unique global strong solution for the initial data in  Sobolev space which do not need a small condition.
\end{abstract}

\vspace{.2in} {\bf Key words:} inhomogeneous magnetohydrodynamic equations;
fractional dissipation; global strong solution.

\vspace{.2in} {\bf MSC(2000):} 35Q35; 35B65; 76N10.



\section{Introduction}
\setcounter{equation}{0} \vskip .1in
 \qquad In this paper, we are interested in studying the
following multi-dimensional inhomogeneous incompressible
magnetohydrodynamic equations with fractional dissipation:
\begin{equation}\label{1.1}
\begin{cases}
\partial_{t}\rho+div(\rho u)=0,\qquad x\in \mathbb{R}^{n},t>0,\\
\partial_{t}(\rho u )+div(\rho u\otimes u)+(-\Delta)^{\alpha}u+\nabla P=(b\cdot\nabla) b, \\
\partial_{t}b+(u\cdot\nabla)b+(-\Delta)^{\beta}b=(b\cdot\nabla)u,\\
\nabla\cdot u=\nabla\cdot b=0, \\
\rho(x,0)=\rho_{0}(x),\qquad u(x,0)=u_{0}(x),\qquad
b(x,0)=b_{0}(x),
\end{cases}
\end{equation}
where $\rho=\rho(x,t)$ represents the density,
$u=(u_{1},u_{2},\cdot\cdot\cdot,u_{n})$  the fluid velocity,
$b=(b_{1},b_{2},\cdot\cdot\cdot,b_{n})$ the magnetic field and
$P=P(x,t)$ the scalar pressure, respectively. Here $\rho_{0},
u_{0},b_{0}$ are  the prescribed  initial data for the density, the
velocity and  the magnetic field with property $\nabla\cdot
u_{0}=\nabla\cdot b_{0}=0.$ The fractional Laplacian operator $(-\Delta)^{\alpha}$  is defined via
the Fourrier transform
$$\widehat{(-\Delta)^{\alpha}f}(\xi)=|\xi|^{2\alpha}\widehat{f}(\xi),$$
where
$$\widehat{f}(\xi)=\frac{1}{(2\pi)^{\frac{d}{2}}}\int_{\mathbb{R}^{d}}e^{-ix\cdot\xi}f(x)\mathrm{d}x.$$
We write $\Lambda=(-\Delta)^{\frac{1}{2}}$ for notational
convenience. The MHD equations are a combination of the
inhomogeneous incompressible Navier-Stokes equation of fluid
dynamics and Maxwell's equation of electromagnetism, where the
displacement current can be neglected (see,
e.g. \cite{K1987,L1996}).
If it is not affected by the magnetic field, namely $b=0$, the
inhomogeneous MHD equations will reduce into inhomogeneous
Navier-Stokes equation, which has been studied by many scholars,
refer to
\cite{AK1972,AK1990,S1990,TA1982,CK2003,CHW2013,HW2015,L2017,WY2018}
and references therein. If the density $\rho$ is a positive
constant, the system \eqref{1.1} reduces to the homogeneous MHD
equations, which have been extensively studied. For instance, G.
Duvaut and J.-L. Lions \cite{DL1972} constructed the Leray-Hopf type
global weak solutions of MHD equations with the finite energy  method. M.
Sermange and R. Teman further discussed the properties of these
solutions \cite{ST1983}. Later, the MHD equations have been
extensively studied, please see
\cite{K1989,W1997,HX2005,QH2004,MYZ2007} and references therein.
In particular, C. Cao, J. Wu and B. Yuan \cite{CWY2014} examined
the global regularity of MHD equations with only fractional
Laplace operators $(-\Delta)^{\beta}, \beta>1$ and no dissipation
in $H^{s}(\mathbb{R}^{2}), s>2$.

For the inhomogeneous fluid, it has been also studied by many
authors. It is known that J. Gerbeau and C. Le Bris \cite{GB1997}
established the global existence of the weak solutions by the
finite energy in 3D bound domains. Global existence of
strong solutions with small initial data in Besov spaces was
considered by Abidi-Paicu \cite{AP2008}. Moreover, under the
natural compatibility condition
$$-\mu\Delta u_{0}+\nabla
P_{0}-(b_{0}\cdot\nabla)b_{0}=\sqrt{\rho_{o}}g,(P_{0},g)\in
H^{1}\times L^{2},$$
Chen, Tan, and Wang obtained the unique local
strong solutions and the global existence of strong solutions
under data satisfy some smallness condition in the whole space
$\mathbb{R}^{3}$. L$\ddot{u}$, Xu and Zhong \cite{LXZ2017} proved
the global existence and large time asymptotic behavior of strong
solutions to the inhomogeneous MHD equations with vacuum. Relevant
problems can also be referred to \cite{HW2013,G2014,ZY2015}.

Recently, focuses have been on the magnetohydrodynamic equations
with partial or fractional dissipation (see,
e.g.,\cite{CW2011,W2003,CWY2014,Y2019,YZ2018}) and references therein. For example, Dehua Wang and Zhuan Ye in
\cite{WY2018} proved that if $\alpha\ge \frac12+\frac{n}{4}$ and
$n\ge 3$ then there exists a  unique strong solution to the
inhomogeneous incompressible Navier-Stokes equations for large
initial data. When $\alpha=\beta=\frac{1}{2}+\frac{n}{4}$, Ye
\cite{Y2019} got global existence of strong solution  to the
multi-dimensional inhomogeneous incompressible MHD equations.

In this paper, inspired by \cite{Y2019}, when
$\alpha+\beta=1+\frac{n}{2}$ satisfying $1\le \beta\le
\alpha\le\min \{\frac{3\beta}{2},\frac{n}{2},1+\frac{n}{4}\}$ and
$\frac{n}{4}<\alpha$ for  $n\geq3$, we obtain the  global existence
of strong solution to the multi-dimensional inhomogeneous
incompressible MHD equations. Our main result is presented as
follows:
~\\
\begin{theorem}\label{thm1}
Consider the magnetohydrodynamic  equations \ref{1.1} with
$\alpha+\beta=1+\frac{n}{2}$ satisfying $1\le \beta\le \alpha\le\min \{\frac{3\beta}{2},\frac{n}{2},1+\frac{n}{4}\}$ and $\frac{n}{4}+\frac{1}{2}\leq\alpha$ for  $n\geq3$. Assume that the initial data
$(\rho_{0},u_{0},b_{0})$ satisfy the following conditions
$$0<c_{0}<\rho_{0}\in
L^{\infty}(\mathbb{R}^{n}), \qquad\nabla \rho_{0}\in
L^{\frac{2n}{6\alpha-n}}(\mathbb{R}^{n}),$$
$$\nabla\cdot u_{0}=\nabla\cdot b_{0}=0,\qquad
u_{0}\in\dot{H}^{\alpha}(\mathbb{R}^{n}),\qquad b_{0}\in
H^{\beta}(\mathbb{R}^{n}),\qquad\sqrt{\rho_{0}}u_{0}\in
L^{2}(\mathbb{R}^{n}),$$
where $c_0$ is a given positive constant.
 Then the magnetohydrodynamic equations
\eqref{1.1} have a unique global strong solution $(\rho,u,b)$
satisfying for any given $0<T<\infty$
\begin{align}\label{1.2}
\begin{cases}
u\in C([0,T];\dot
H^{\alpha}(\mathbb{R}^{n})),\rho u\in
C([0,T];L^2(\mathbb{R}^{n})), b\in C([0,T];
H^{\beta}(\mathbb{R}^{n})),\\
\rho\in C([0,T];L^q(\mathbb{R}^{n})),
\frac{n}{2\alpha}\leq q<\infty,\\
c_0<\rho\in L^{\infty}(0,T;L^{\infty}(\mathbb{R}^{n})),\nabla\rho\in
L^{\infty}(0,T;L^{\frac{2n}{6\alpha-n}}(\mathbb{R}^{n})),\\
\sqrt{\rho}u,\Lambda^{\alpha}u,t^{\frac{1}{2}}\Lambda^{\alpha}u,\sqrt{\rho}\partial_{t}u\in
L^{\infty}(0,T;L^{2}(\mathbb{R}^{n})),\\
b,\Lambda^{\beta}b,t^{\frac{1}{2}}\Lambda^{\beta}b,\partial_{t}b\in
L^{\infty}(0,T;L^{2}(\mathbb{R}^{n})),\\
\Lambda^{\alpha}u,\Lambda^{2\alpha}u,\sqrt{\rho}\partial_{t}u,\Lambda^{\alpha}\partial_{t}u,\nabla
P\in L^{2}(0,T;L^{2}(\mathbb{R}^{n})),\\
\Lambda^{\beta}b,\Lambda^{2\beta}b,\partial_{t}b,\Lambda^{\beta}\partial_{t}b\in
L^{2}(0,T;L^{2}(\mathbb{R}^{n})).
\end{cases}
\end{align}
Moreover, the solution $(\rho,u,b)$
admits the following decay rates for all $t> 0$
$$\|\Lambda^{\alpha}u\|_{L^{2}}+\|\Lambda^{\beta}b\|_{L^{2}}+\|\sqrt{\rho}\partial_{t}u(t)\|_{L^{2}}+\|\partial_{t}b(t)\|_{L^{2}}\leq
C_{0}t^{-\frac{1}{2}},$$
where $C_{0}$ depends on
$\|\rho_{0}\|_{L^{\infty}},\|\sqrt{\rho_{0}}u_{0}\|_{L^{2}},\|u_{0}\|_{\dot{H}^{\alpha}},\|b_{0}\|_{H^{\beta}}$.
\end{theorem}

\begin{remark}\label{remark1}
For the inhomogeneous incompressible Navier-Stokes equation,
Dehua Wang and Zhuan Ye in \cite{WY2018} proved that if $\alpha\ge
\frac12+\frac{n}{4}$ and $n\ge 3$ then there exists a unique global
strong solution for large initial data.  For the case $\alpha=\beta=\frac{1}{2}+\frac{n}{4}$ Zhuan Ye in \cite{Y2019} obtained the unique global strong solution to the inhomogeneous
MHD equations (\ref{1.1}) for large initial data. The novelty of this paper
is that we establish the unique global strong solution under the conditions $\alpha +\beta=1+\frac{n}{2}$,
$1\le \beta\le \alpha\le\min
\{\frac{3\beta}{2},\frac{n}{2},1+\frac{n}{4}\}$ and
$\frac{n}{4}<\alpha$, which is that $\alpha$ and $\beta$ are not
too far but not necessarily equal. when
$\alpha=\beta=\frac{1}{2}+\frac{n}{4}$ we obtain the result of
\cite{Y2019}, therefor Theorem \ref{thm1} generalizes the results
of \cite{Y2019} to the case $\alpha\neq\beta$, and we require the
initial density is bounded below.
\end{remark}

To prove the global well-posedness Theorem \ref{thm1}, we need the local well-posedness result of strong solutions which can
 be proved  by a standard procedure of approximate scheme. For the proof of this proposition, readers can refer to \cite{Y2019}.

\begin{proposition}\label{proposition1}
$\mathrm{(local\ strong\ solution)}$. Under the assumptions of Theorem
\ref{thm1}. For the initial data $(\rho_0(x), u_0(x), b_0(x))$, there exists a small time $T^{*}$ dependent on $\|\rho_{0}\|_{L^{\infty}},\|\nabla \rho_{0}\|_{
L^{\frac{2n}{6\alpha-n}}},\|\sqrt{\rho_{0}}u_{0}\|_{L^{2}},\|u_{0}\|_{\dot{H}^{\alpha}},\|b_{0}\|_{H^{\beta}}$ such that the MHD equations \eqref{1.1} have a unique strong solution $(\rho, u, b)$ satisfying \eqref{1.2} on $[0,T^*)$.
\end{proposition}

The rest of this paper is organized in four sections. In Sect. 2,
we introduce several important lemmas, which will be used in the paper. In Sect. 3, we give some a priori estimates for the local solutions. In Sect. 4, we are devoted to
proving Theorem \ref{thm1}. Throughout this paper, we will use $C$ denote a finite inessential constant which may be different from line to line, but do not depend on particular solutions or function.

\section{Preliminaries}
\setcounter{equation}{0}

In this section, we
 give some lemmas. The first one is the classical Gr\"{o}nwall
inequality, which will be used frequently. The second  lemma
is used to prove the uniqueness of strong solutions.


\begin{lemma}\label{lemma1}
Let $X(t)$, $Y(t)$, $\beta(t)$ and $\gamma(t)$ are non-negative
functions, and $\beta(t)$ and $\gamma(t)$ are two
integrable functions over $[a, b]$. If the following differential
inequality holds
$$\frac{\mathrm{d}}{\mathrm{d}t}X(t)+Y(t)\leq\beta(t)+\gamma(t)X(t), \qquad a\leq t
\leq b,$$ then
\begin{align*}
X(t)+\int_{a}^{t}Y(s)\mathrm{d}s\leq
\bigg(X(a)+\int_{a}^{t}\beta(s)\mathrm{d}s\bigg)e^{\int_{a}^{\eta}\gamma(s)\mathrm{d}s}.
\end{align*}
\end{lemma}


\begin{lemma}\label{2.2}$\mathrm{(\cite{L1999})}$
Let $X(t)$, $Y(t)$, $Z(t)$, $\gamma(t)$ and $\eta(t)$ are
non-negative functions. Let $X(t)$ and $Y(t)$ are absolutely
continuous on $[0,T]$ and satisfy
\begin{align*}
\begin{cases}
\frac{\mathrm{d}}{\mathrm{d}t}X(t)\leq AZ^{\frac{1}{2}}(t),\\
\frac{\mathrm{d}}{\mathrm{d}t}Y(t)+Z(t)\leq\gamma(t)Y(t)+\eta(t)X^{2}(t), \\
X(0)=0, \\
\end{cases}
\end{align*}
where A is a positive constant, $\gamma(t)$ and $\eta(t)$ are two
integrable functions over $[0,T]$. Then, the following estimates
hold $$X(t)\leq
AY^{\frac{1}{2}}(0)t^{\frac{1}{2}}e^{\frac{1}{2}\int_{0}^{t}(\gamma(s)+A^{2}s\eta(s))\mathrm{d}s},$$
$$Y(t)+\int_{0}^{t}Z(s)\mathrm{d}s\leq
Y(0)e^{\int_{0}^{t}(\gamma(s)+A^{2}s\eta(s))\mathrm{d}s}.$$ In
particular, if $Y(0)=0$, then we get $$X(t)=Y(t)=Z(t)=0.$$
\end{lemma}


\section{A priori estimates}
\setcounter{equation}{0}

 In this section, we establish
some necessary  bounds for a strong solution. The
following Lemma \ref{lemma3.1} shows the estimates of $\|\rho(t)\|_{
L^{\infty}(0,T;L^{\infty})}$ and $\|(\sqrt{\rho}u,b)(t)\|_{L^{\infty}(0,T;L^{2})}$.


\begin{lemma}\label{lemma3.1}
Under the assumptions of Theorem \ref{thm1}, the corresponding
solution $(\rho,u,b)$ of the equations \eqref{1.1} admits the
following bounds for any $t>0$
\begin{align}\label{3.1}
\|\rho(t)\|_{
L^{\infty}}\leq\|\rho_{0}\|_{
L^{\infty}},
\end{align}
\begin{align}\label{3.2}
\|(\sqrt{\rho}u,b)(t)\|_{L^{2}}^{2}+2\int_{0}^{t}\|(\Lambda^{\alpha}u,\Lambda^{\beta}b)(\tau)\|_{L^{2}}^{2}\mathrm{d}\tau\leq\|(\sqrt{\rho_{0}}u_{0},b_{0})\|_{L^{2}}^{2}.
\end{align}
\end{lemma}


\begin{proof}
We note that \eqref{3.1} follows from the property of transport equation
$\eqref{1.1}_{1}$ and using the divergence free condition $\eqref{1.1}_{4}$. To prove \eqref{3.2}, multiplying the equation
$\eqref{1.1}_{2,3}$ by $(u,b)$, adding the results together and integrating by parts, we have
\begin{align}\label{3.3}
\frac{1}{2}\frac{\mathrm{d}}{\mathrm{d}t}\|(\sqrt{\rho}u,b)(t)\|_{L^{2}}^{2}+\|(\Lambda^{\alpha}u,\Lambda^{\beta}b)(t)\|_{L^{2}}^{2}=0.
\end{align}
 Integrating \eqref{3.3} over $[0,t]$, we arrive at
\eqref{3.2}.
\end{proof}

Next, we calculate the estimates in the
$L^{\infty}(0,T;\dot{H}^{\alpha}(\mathbb{R}^{n}))$-norm of the
velocity field and the $L^{\infty}(0,T;\dot H^{\beta}(\mathbb{R}^{n}))$-norm
of the
magnetic field.


\begin{lemma}\label{lemma3.2}
Under the assumptions of Theorem \ref{thm1}, the corresponding
solution $(\rho,u,b)$ of the equations \eqref{1.1} admits the
following bounds for any $t>0$
\begin{align}\label{3.4}
\|(\Lambda^{\alpha}u,\Lambda^{\beta}b)(t)\|_{L^{2}}^{2}+\|\Lambda^{\frac{\beta}{2}}b(t)\|_{L^{2}}^{4}+\int_{0}^{t}\|(\Lambda^{2\alpha}u,\Lambda^{2\beta}b,\sqrt{\rho}\partial_{\tau}u,\partial_{\tau}b)(\tau)\|_{L^{2}}^{2}\mathrm{d}\tau\leq
C_{0},
\end{align}
\begin{align}\label{3.5}
t\|(\Lambda^{\alpha}u,\Lambda^{\beta}b)(t)\|_{L^{2}}^{2}+t\|\Lambda^{\frac{\beta}{2}}b(t)\|_{L^{2}}^{4}+\int_{0}^{t}\tau\|(\Lambda^{2\alpha}u,\Lambda^{2\beta}b,\sqrt{\rho}\partial_{\tau}u,\partial_{\tau}b)(\tau)\|_{L^{2}}^{2}\mathrm{d}\tau\leq
C_{0},
\end{align}
where $C_{0}$ depends only on
$\|\rho_0\|_{L^{\infty}},\|\sqrt{\rho_{0}}u_{0}\|_{L^{2}},\|\Lambda^{\alpha}u_{0}\|_{L^{2}},\|b_{0}\|_{H^{\beta}},$
and  $C_{0}(t)$ depends only on the initial data and time $t$.
\end{lemma}


\begin{proof}
First, multiplying $\eqref{1.1}_{2}$ by $\partial_{t}u$  and  integrating by parts, we get
\begin{align}\label{3.6}
\frac{1}{2}\frac{\mathrm{d}}{\mathrm{d}t}\|\Lambda^{\alpha}u(t)\|_{L^{2}}^{2}+\|\sqrt{\rho}\partial_{t}u\|_{L^{2}}^{2}=-\int_{\mathbb{R}^{n}}(\rho
u\cdot\nabla u)\cdot\partial_{t}u\mathrm{d}x+\int_{\mathbb{R}^{n}}
(b\cdot\nabla b)\cdot\partial_{t}u\mathrm{d}x.
\end{align}
Second,  taking the $L^{2}$ inner product of  $\eqref{1.1}_{3}$ with $\partial_{t}b$ and integrating by parts, we obtain
\begin{align}\label{3.7}
\frac{1}{2}\frac{\mathrm{d}}{\mathrm{d}t}\|\Lambda^{\beta}b(t)\|_{L^{2}}^{2}+\|\partial_{t}b\|_{L^{2}}^{2}=-\int_{\mathbb{R}^{n}}
(u\cdot\nabla
b)\cdot\partial_{t}b\mathrm{d}x+\int_{\mathbb{R}^{n}}
(b\cdot\nabla u)\cdot\partial_{t}b\mathrm{d}x.
\end{align}
Then, adding the equalities \eqref{3.6} and \eqref{3.7} together, we have
\begin{align}\label{3.8}
\frac{1}{2}\frac{\mathrm{d}}{\mathrm{d}t}\|(\Lambda^{\alpha}u,\Lambda^{\beta}b)(t)\|_{L^{2}}^{2}
+\|(\sqrt{\rho}\partial_{t}u,\partial_{t}b)\|_{L^{2}}^{2}=\sum_{i=1}^{4}I_{i},
\end{align}
where $$I_{1}=-\int_{\mathbb{R}^{n}}(\rho u\cdot\nabla
u)\cdot\partial_{t}u\mathrm{d}x,\qquad I_{2}=\int_{\mathbb{R}^{n}}
(b\cdot\nabla b)\cdot\partial_{t}u\mathrm{d}x,$$
$$I_{3}=-\int_{\mathbb{R}^{n}}
(u\cdot\nabla b)\cdot\partial_{t}b\mathrm{d}x,\qquad
I_{4}=\int_{\mathbb{R}^{n}} (b\cdot\nabla
u)\cdot\partial_{t}b\mathrm{d}x.$$

We use H\"older,
 Gagliardo-Nirenberg's inequalities and the a priori estimates \eqref{3.1},
\eqref{3.2} to estimate each term on the right-hand side of
\eqref{3.8} as follows:
\begin{align}\label{3.9}
\mid I_{1}\mid &\leq
C\|\sqrt{\rho}\|_{L^{\infty}}\|\sqrt{\rho}\partial_{t}u\|_{L^{2}}\|u\cdot\nabla
u\|_{L^{2}}\notag\\
&\leq
C\|\rho_{0}\|_{L^{\infty}}^{\frac{1}{2}}\|\sqrt{\rho}\partial_{t}u\|_{L^{2}}\|u\|_{L^{\frac{2n}{n-2\alpha}}}\|\nabla
u\|_{L^{\frac{n}{\alpha}}}\notag\\
&\leq
C\|\sqrt{\rho}\partial_{t}u\|_{L^{2}}\|\Lambda^{\alpha}u\|_{L^{2}}\|\Lambda^{\beta}
u\|_{L^{2}}\notag\\
&\leq
C\|\sqrt{\rho}\partial_{t}u\|_{L^{2}}\|\Lambda^{\alpha}u\|_{L^{2}}(\|u\|_{L^{2}}^{1-\frac{\beta}{\alpha}}\|\Lambda^{\alpha}
u\|_{L^{2}}^{\frac{\beta}{\alpha}})\notag\\
&\leq
C\|\sqrt{\rho}\partial_{t}u\|_{L^{2}}\|\Lambda^{\alpha}u\|_{L^{2}}(\|u\|_{L^{2}}+\|\Lambda^{\alpha}
u\|_{L^{2}})\notag\\
&\leq
C\|\sqrt{\rho}\partial_{t}u\|_{L^{2}}\|\Lambda^{\alpha}u\|_{L^{2}}(\|\sqrt{\rho}u\|_{L^{2}}+\|\Lambda^{\alpha}
u\|_{L^{2}})\notag\\
&\leq
\frac{1}{2}\|\sqrt{\rho}\partial_{t}u\|_{L^{2}}^{2}+C\|\Lambda^{\alpha}u\|_{L^{2}}^{2}(1+\|\Lambda^{\alpha}
u\|_{L^{2}}^{2}).
\end{align}

By a similar argument as the estimate (\ref{3.9}), we have
\begin{align}\label{3.11}
\mid I_{3}\mid+\mid I_{4}\mid&\leq
C\|\partial_{t}b\|_{L^{2}}(\|\nabla
b\|_{L^{\frac{n}{\alpha}}}\|u\|_{L^{\frac{2n}{n-2\alpha}}}+\|\nabla
u\|_{L^{\frac{n}{\beta}}}\|b\|_{L^{\frac{2n}{n-2\beta}}})\notag\\
&\leq C\|\partial_{t}b\|_{L^{2}}(\|\Lambda^{\beta}
b\|_{L^{2}}\|\Lambda^{\alpha}u\|_{L^{2}}+\|\Lambda^{\alpha}
u\|_{L^{2}}\|\Lambda^{\beta}b\|_{L^{2}})\notag\\
&\leq \frac{1}{4}\|\partial_{t}b\|_{L^{2}}^{2}+C\|\Lambda^{\alpha}
u\|_{L^{2}}^{2}\|\Lambda^{\beta}b\|_{L^{2}}^{2}.
\end{align}

For the term $I_2$, by a simple calculation to transform the derivative with respect to $t$ from $u(t)$ to $b(t)$, and the similar argument yields

\begin{align}\label{3.10}
 I_{2}&=\frac{\mathrm{d}}{\mathrm{d}t}\int_{\mathbb{R}^{n}}(b\cdot\nabla
b)\cdot
u\mathrm{d}x-\int_{\mathbb{R}^{n}}(\partial_{t}b\cdot\nabla
b)\cdot
u\mathrm{d}x-\int_{\mathbb{R}^{n}}(b\cdot\nabla\partial_{t}b)\cdot
u\mathrm{d}x\notag\\
&=\frac{\mathrm{d}}{\mathrm{d}t}\int_{\mathbb{R}^{n}}(b\cdot\nabla b)\cdot
u\mathrm{d}x-\int_{\mathbb{R}^{n}}(\partial_{t}b\cdot\nabla
b)\cdot
u\mathrm{d}x-\int_{\mathbb{R}^{n}}b_{i}\partial_{i}\partial_{t}b_{j}
u_{j}\mathrm{d}x\notag\\
&=\frac{\mathrm{d}}{\mathrm{d}t}\int_{\mathbb{R}^{n}}(b\cdot\nabla b)\cdot
u\mathrm{d}x-\int_{\mathbb{R}^{n}}(\partial_{t}b\cdot\nabla
b)\cdot u\mathrm{d}x+\int_{\mathbb{R}^{n}}b_{i}\partial_{t}b_{j}
\partial_{i}u_{j}\mathrm{d}x\notag\\
&\leq\frac{\mathrm{d}}{\mathrm{d}t}\int_{\mathbb{R}^{n}}(b\cdot\nabla b)\cdot
u\mathrm{d}x+C\|\partial_{t}b\|_{L^{2}}\|\nabla
b\|_{L^{\frac{n}{\alpha}}}\|u\|_{L^{\frac{2n}{n-2\alpha}}}+C\|\partial_{t}b\|_{L^{2}}\|\nabla
u\|_{L^{\frac{n}{\beta}}}\|b\|_{L^{\frac{2n}{n-2\beta}}}\notag\\
&\leq\frac{\mathrm{d}}{\mathrm{d}t}\int_{\mathbb{R}^{n}}(b\cdot\nabla b)\cdot
u\mathrm{d}x+C\|\partial_{t}b\|_{L^{2}}(\|\Lambda^{\beta}
b\|_{L^{2}}\|\Lambda^{\alpha}u\|_{L^{2}}+\|\Lambda^{\alpha}
u\|_{L^{2}}\|\Lambda^{\beta}b\|_{L^{2}})\notag\\
&\leq\frac{\mathrm{d}}{\mathrm{d}t}\int_{\mathbb{R}^{n}}(b\cdot\nabla b)\cdot
u\mathrm{d}x+\frac{1}{4}\|\partial_{t}b\|_{L^{2}}^{2}+C\|\Lambda^{\alpha}
u\|_{L^{2}}^{2}\|\Lambda^{\beta}b\|_{L^{2}}^{2}.
\end{align}

Inserting estimates $\eqref{3.9}-\eqref{3.10}$ into \eqref{3.8}, we obtain
\begin{align}\label{3.12}
\frac{\mathrm{d}}{\mathrm{d}t}(\|(\Lambda^{\alpha}u,\Lambda^{\beta}b)(t)\|_{L^{2}}^{2}+M(t))
+\|(\sqrt{\rho}\partial_{t}u,\partial_{t}b)\|_{L^{2}}^{2}\leq
C\|\Lambda^{\alpha} u\|_{L^{2}}^{2}(\|(\Lambda^{\alpha}
u,\Lambda^{\beta}b)\|_{L^{2}}^{2}+1),
\end{align}
where, by integrating by parts, we use
$$\frac{\mathrm{d}}{\mathrm{d}t}\int_{\mathbb{R}^{n}}(b\cdot\nabla b)\cdot
u\mathrm{d}x=-\frac{\mathrm{d}}{\mathrm{d}t}\int_{\mathbb{R}^{n}}(b\cdot\nabla
u)\cdot b\mathrm{d}x,$$
and we denote
$$M(t)\triangleq\int_{\mathbb{R}^{n}}(b\cdot\nabla
u)\cdot b\mathrm{d}x.$$
 On the one hand, we can obtain by direct
calculation
\begin{align}\label{3.13}
\mid M(t)\mid&\leq\|b\|_{L^{2}}\|b\cdot\nabla u\|_{L^{2}}\notag\\
&\leq C\|b_{0}\|_{L^{2}}\|\nabla
u\|_{L^{\frac{n}{\beta}}}\|b\|_{L^{\frac{2n}{n-2\beta}}}\notag\\
&\leq C(\|\Lambda^{\alpha}
u\|_{L^{2}}^{2}+\|\Lambda^{\beta}b\|_{L^{2}}^{2}).
\end{align}
On the other hand, we also obtain
\begin{align*}
\mid M(t)\mid&\leq\|b\|_{L^{\frac{2n}{n-\beta}}}^{2}\|\nabla u\|_{L^{\frac{n}{\beta}}}\\
&\leq C\|\Lambda^{\alpha}
u\|_{L^{2}}\|\Lambda^{\frac{\beta}{2}}b\|_{L^{2}}^{2}\\
&\leq \frac{1}{2}\|\Lambda^{\alpha}
u\|_{L^{2}}^{2}+C\|\Lambda^{\frac{\beta}{2}}b\|_{L^{2}}^{4},
\end{align*}
which immediately gives
\begin{align}\label{3.14}
M(t)\geq-\frac{1}{2}\|\Lambda^{\alpha}
u\|_{L^{2}}^{2}-\tilde{C}\|\Lambda^{\frac{\beta}{2}}b\|_{L^{2}}^{4}.
\end{align}
In order to close the above inequality, we need to deal with
$\|\Lambda^{\frac{\beta}{2}}b\|_{L^{2}}$. We apply the operator
$\Lambda^{\frac{\beta}{2}}$ on the both sides of the equation $\eqref{1.1}_{3}$ and
multiply the result equation by $\Lambda^{\frac{\beta}{2}}b$ to deduce that
\begin{align}\label{3.15}
\frac{1}{2}\frac{\mathrm{d}}{\mathrm{d}t}\|\Lambda^{\frac{\beta}{2}}b(t)\|_{L^{2}}^{2}+\|\Lambda^{\frac{3\beta}{2}}b\|_{L^{2}}^{2}=-\int_{\mathbb{R}^{n}}
\Lambda^{\frac{\beta}{2}}(u\cdot\nabla
b)\cdot\Lambda^{\frac{\beta}{2}}b\mathrm{d}x+\int_{\mathbb{R}^{n}}
\Lambda^{\frac{\beta}{2}}(b\cdot\nabla
u)\cdot\Lambda^{\frac{\beta}{2}}b\mathrm{d}x.
\end{align}
The first term at the right hand side of \eqref{3.15} can be estimated as follows by the H\"{o}lder and Sobolev inequalities,
\begin{align}\label{3.16}
\int_{\mathbb{R}^{n}} \Lambda^{\frac{\beta}{2}}(u\cdot\nabla
b)\cdot\Lambda^{\frac{\beta}{2}}b\mathrm{d}x&\leq C\|u\cdot\nabla
b\|_{L^{2}}\|\Lambda^{\beta}b\|_{L^{2}}\notag\\
&\leq C\|u\|_{L^{\frac{2n}{n-2\alpha}}}\|\nabla b\|_{L^{\frac{n}{\alpha}}}\|\Lambda^{\beta}b\|_{L^{2}}\notag\\
&\leq
C\|\Lambda^{\alpha}u\|_{L^{2}}\|\Lambda^{\beta}b\|_{L^{2}}^{2}\notag\\
&\leq C\|\Lambda^{\alpha}u\|_{L^{2}}\|\Lambda^{\frac{\beta}{2}}b\|_{L^{2}}\|\Lambda^{\frac{3\beta}{2}}b\|_{L^{2}}\notag\\
&\leq
\frac{1}{4}\|\Lambda^{\frac{3\beta}{2}}b\|_{L^{2}}^{2}+C\|\Lambda^{\alpha}u\|_{L^{2}}^{2}\|\Lambda^{\frac{\beta}{2}}b\|_{L^{2}}^{2}.
\end{align}
The second term at the right hand side of \eqref{3.15} is similarly estimated as follows
\begin{align}\label{3.17}
\int_{\mathbb{R}^{n}} \Lambda^{\frac{\beta}{2}}(b\cdot\nabla
u)\cdot\Lambda^{\frac{\beta}{2}}b\mathrm{d}x&\leq C\|b\cdot\nabla
u\|_{L^{\frac{2n}{n+\beta}}}\|\Lambda^{\beta}b\|_{L^{\frac{2n}{n-\beta}}}\notag\\
&\leq C\|b\|_{L^{\frac{2n}{n-\beta}}}\|\nabla
u\|_{L^{\frac{n}{\beta}}}\|\Lambda^{\frac{3\beta}{2}}b\|_{L^{2}}\notag\\
&\leq
C\|\Lambda^{\frac{\beta}{2}}b\|_{L^{2}}\|\Lambda^{\alpha}u\|_{L^{2}}\|\Lambda^{\frac{3\beta}{2}}b\|_{L^{2}}\notag\\
&\leq
\frac{1}{4}\|\Lambda^{\frac{3\beta}{2}}b\|_{L^{2}}^{2}+C\|\Lambda^{\frac{\beta}{2}}b\|_{L^{2}}^{2}\|\Lambda^{\alpha}u\|_{L^{2}}^{2}.
\end{align} Then, inserting inequalities
\eqref{3.16}-\eqref{3.17} into \eqref{3.15}, we have
\begin{align*}
\frac{\mathrm{d}}{\mathrm{d}t}\|\Lambda^{\frac{\beta}{2}}b(t)\|_{L^{2}}^{2}+\|\Lambda^{\frac{3\beta}{2}}b\|_{L^{2}}^{2}\leq
C\|\Lambda^{\alpha}u\|_{L^{2}}^{2}\|\Lambda^{\frac{\beta}{2}}b\|_{L^{2}}^{2}.
\end{align*}
By the a priori estimate \eqref{3.2} and the Gr\"{o}nwall
Lemma \ref{lemma1}, we have
\begin{align}\label{3.18}
\|\Lambda^{\frac{\beta}{2}}b(t)\|_{L^{2}}^{2}+\int_{0}^{t}\|\Lambda^{\frac{3\beta}{2}}b(\tau)\|_{L^{2}}^{2}\mathrm{d}\tau\leq
C_{0}.
\end{align}
Moreover, we also have
\begin{align}\label{3.19}
\frac{\mathrm{d}}{\mathrm{d}t}\|\Lambda^{\frac{\beta}{2}}b(t)\|_{L^{2}}^{4}\leq
C\|\Lambda^{\alpha}u\|_{L^{2}}^{2}\|\Lambda^{\frac{\beta}{2}}b\|_{L^{2}}^{4}.
\end{align}
 Multiplying \eqref{3.19} by $\tilde{C}+1$, we get
\begin{align}\label{3.20}
\frac{\mathrm{d}}{\mathrm{d}t}((\tilde{C}+1)\|\Lambda^{\frac{\beta}{2}}b(t)\|_{L^{2}}^{4})\leq
(\tilde{C}+1)\|\Lambda^{\alpha}u\|_{L^{2}}^{2}\|\Lambda^{\frac{\beta}{2}}b\|_{L^{2}}^{4}.
\end{align}
Summing up the estimates \eqref{3.12} and \eqref{3.20}, we have
\begin{align}\label{3.21}
&\quad\frac{\mathrm{d}}{\mathrm{d}t}(\|(\Lambda^{\alpha}u,\Lambda^{\beta}b)(t)\|_{L^{2}}+M(t)+(\tilde{C}+1)\|\Lambda^{\frac{\beta}{2}}b(t)\|_{L^{2}}^{4})
+\|(\sqrt{\rho}\partial_{t}u,\partial_{t}b)\|_{L^{2}}^{2}\notag\\&\leq
C\|\Lambda^{\alpha} u\|_{L^{2}}^{2}(\|(\Lambda^{\alpha}
u,\Lambda^{\beta}b)\|_{L^{2}}^{2}+1+\|\Lambda^{\frac{\beta}{2}}b\|_{L^{2}}^{4}).
\end{align}
Next, using the lower bound \eqref{3.14} and integrating $\eqref{3.21}$ with respect to time $t$,
we obtain
\begin{align*}
&\quad\|(\Lambda^{\alpha}u,\Lambda^{\beta}b)(t)\|_{L^{2}}^{2}+\|\Lambda^{\frac{\beta}{2}}b(t)\|_{L^{2}}^{4}
+\int_{0}^{t}\|(\sqrt{\rho}\partial_{t}u,\partial_{t}b)(\tau)\|_{L^{2}}^{2}\mathrm{d}\tau\notag\\&\leq
C_{0}+\int_{0}^{t}C\|\Lambda^{\alpha}
u(\tau)\|_{L^{2}}^{2}(\|(\Lambda^{\alpha}u,\Lambda^{\beta}b)(\tau)\|_{L^{2}}^{2}+1+\|\Lambda^{\frac{\beta}{2}}b(\tau)\|_{L^{2}}^{4})\mathrm{d}\tau.
\end{align*}
Using the Gr\"{o}nwall Lemma \ref{lemma1} and the a priori estimate \eqref{3.2}, it yields
\begin{align}\label{3.22}
\|(\Lambda^{\alpha}u,\Lambda^{\beta}b)(t)\|_{L^{2}}^{2}+1+\|\Lambda^{\frac{\beta}{2}}b(t)\|_{L^{2}}^{4}
+\int_{0}^{t}\|(\sqrt{\rho}\partial_{t}u,\partial_{t}b)(\tau)\|_{L^{2}}^{2}\leq
C_{0}.
\end{align}
We can also obtain from estimate \eqref{3.21} by multiplying time $t$ that
\begin{align*}
&\quad\frac{\mathrm{d}}{\mathrm{d}t}(t\|(\Lambda^{\alpha}u,\Lambda^{\beta}b)\|_{L^{2}}^{2}+1+tM(t)+t(\tilde{C}+1)\|\Lambda^{\frac{\beta}{2}}b(t)\|_{L^{2}}^{4})
+t\|(\sqrt{\rho}\partial_{t}u,\partial_{t}b)\|_{L^{2}}^{2}\notag\\&\leq
\|(\Lambda^{\alpha}u,\Lambda^{\beta}b)\|_{L^{2}}^{2}+M(t)+(\tilde{C}+1)\|\Lambda^{\frac{\beta}{2}}b\|_{L^{2}}^{4}+Ct\|\Lambda^{\alpha} u\|_{L^{2}}^{2}(\|(\Lambda^{\alpha}u,\Lambda^{\beta}b)\|_{L^{2}}^{2}+1+\|\Lambda^{\frac{\beta}{2}}b\|_{L^{2}}^{4})\notag\\
&\leq
C\|(\Lambda^{\alpha}u,\Lambda^{\beta}b)\|_{L^{2}}^{2}+(\tilde{C}+1)\|\Lambda^{\frac{\beta}{2}}b\|_{L^{2}}^{4}+Ct\|\Lambda^{\alpha} u\|_{L^{2}}^{2}(\|(\Lambda^{\alpha}u,\Lambda^{\beta}b)\|_{L^{2}}^{2}+1+\|\Lambda^{\frac{\beta}{2}}b\|_{L^{2}}^{4}),
\end{align*}
where in the last line we have applied the estimate \eqref{3.13}. Integrating
the above inequality with respect to time $t$, we have
\begin{align*}
&\quad t\|(\Lambda^{\alpha}u,\Lambda^{\beta}b)(t)\|_{L^{2}}^{2}+t+tM(t)+t(\tilde{C}+1)\|\Lambda^{\frac{\beta}{2}}b(t)\|_{L^{2}}^{4}+\int_{0}^{t}\tau\|(\sqrt{\rho}\partial_{\tau}u,\partial_{\tau}b)(\tau)\|_{L^{2}}^{2}\mathrm{d}\tau\notag\\&\leq
C_{0}+C\int_{0}^{t}\tau\|\Lambda^{\alpha}
u(\tau)\|_{L^{2}}^{2}(\|(\Lambda^{\alpha}u,\Lambda^{\beta}b)(\tau)\|_{L^{2}}^{2}+1+\|\Lambda^{\frac{\beta}{2}}b(\tau)\|_{L^{2}}^{4})\mathrm{d}\tau,
\end{align*}
where we used a Sobolev interpolation inequality,
$$\int_{0}^{t}\|\Lambda^{\frac{\beta}{2}}b(\tau)\|_{L^{2}}^{4}\mathrm{d}\tau\leq\int_{0}^{t}\|b(\tau)\|_{L^{2}}^{2}\|\Lambda^{\beta}b(\tau)\|_{L^{2}}^{2}\mathrm{d}\tau\leq
C_{0}.$$
Applying of the lower bound \eqref{3.14}, we can get
\begin{align*}
&\quad t\|(\Lambda^{\alpha}u,\Lambda^{\beta}b)(t)\|_{L^{2}}^{2}+t+t\|\Lambda^{\frac{\beta}{2}}b(t)\|_{L^{2}}^{4}+\int_{0}^{t}\tau\|(\sqrt{\rho}\partial_{\tau}u,\partial_{\tau}b)(\tau)\|_{L^{2}}^{2}\mathrm{d}\tau\notag\\&\leq
C_{0}+C\int_{0}^{t}\tau\|\Lambda^{\alpha}
u(\tau)\|_{L^{2}}^{2}(\|(\Lambda^{\alpha}u,\Lambda^{\beta}b)(\tau)\|_{L^{2}}^{2}+1+\|\Lambda^{\frac{\beta}{2}}b(\tau)\|_{L^{2}}^{4})\mathrm{d}\tau.
\end{align*}
The Gr\"{o}nwall Lemma \ref{lemma1} and the a priori estimate \eqref{3.2} allow us to deduce
\begin{align}\label{3.23}
t(\|(\Lambda^{\alpha}u,\Lambda^{\beta}b)(t)\|_{L^{2}}^{2}+1+\|\Lambda^{\frac{\beta}{2}}b(t)\|_{L^{2}}^{4})+\int_{0}^{t}\tau\|(\sqrt{\rho}\partial_{\tau}u,\partial_{\tau}b)(\tau)\|_{L^{2}}^{2}\mathrm{d}\tau\leq C_{0}.
\end{align}
Let us recall the Stokes equations:
\begin{equation}\label{3.24}
\begin{cases}
(-\Delta)^{\alpha}u+\nabla P =(b\cdot\nabla) b-\rho\partial_{t}u -(\rho u\cdot\nabla) u, \\
\nabla\cdot u=0.
\end{cases}
\end{equation}
It is easy to deduce from \eqref{3.24} that
\begin{align}\label{u2alpha}
\|\Lambda^{2\alpha}u\|_{L^{2}}&\leq C\|b\cdot\nabla
b\|_{L^{2}}+C\|\rho\partial_{t}u\|_{L^{2}}+\|\rho u\cdot\nabla
u\|_{L^{2}} \\ \notag
&\leq C\|b\|_{L^{\frac{2n}{n-2\beta}}}\|\nabla
b\|_{L^{\frac{n}{\beta}}}+C\|\sqrt{\rho}\|_{L^{\infty}}\|\sqrt{\rho}\partial_{t}u\|_{L^{2}}+C\|\rho\|_{L^{\infty}}\|u\|_{L^{\frac{2n}{n-2\alpha}}}\|\nabla
u\|_{L^{\frac{n}{\alpha}}}\\ \notag
&\leq
C\|\Lambda^{\beta}b\|_{L^{2}}\|\Lambda^{\alpha}b\|_{L^{2}}+C\|\sqrt{\rho}\partial_{t}u\|_{L^{2}}+C\|\Lambda^{\alpha}u\|_{L^{2}}\|\Lambda^{\beta}
u\|_{L^{2}}\\ \notag
&\leq
C\|\Lambda^{\beta}b\|_{L^{2}}\|b\|_{L^{2}}^{1-\frac{2\alpha}{3\beta}}\|\Lambda^{\frac{3\beta}{2}}b\|_{L^{2}}^{\frac{2\alpha}{3\beta}}+C\|\sqrt{\rho}\partial_{t}u\|_{L^{2}}+C\|\Lambda^{\alpha}u\|_{L^{2}}\|\Lambda^{\alpha}
u\|_{L^{2}}^{\frac{\beta}{\alpha}}\|u\|_{L^{2}}^{1-\frac{\beta}{\alpha}}\\ \notag
&\leq
C\|\Lambda^{\beta}b\|_{L^{2}}(\|b\|_{L^{2}}+\|\Lambda^{\frac{3\beta}{2}}b\|_{L^{2}})+C\|\sqrt{\rho}\partial_{t}u\|_{L^{2}}+C\|\Lambda^{\alpha}u\|_{L^{2}}(\|\Lambda^{\alpha}u\|_{L^{2}}+\|u\|_{L^{2}}),
\end{align}
which implies
\begin{align*}
\int_{0}^{t}\tau\|\Lambda^{2\alpha}u(\tau)\|_{L^{2}}^{2}\mathrm{d}\tau
&\leq
C\int_{0}^{t}\tau\|\Lambda^{\beta}b(\tau)\|_{L^{2}}^{2}(\|b(\tau)\|_{L^{2}}^{2}+\|\Lambda^{\frac{3\beta}{2}}b(\tau)\|_{L^{2}}^{2})\mathrm{d}\tau+C\int_{0}^{t}\tau\|\sqrt{\rho}\partial_{t}u\|_{L^{2}}^{2}\mathrm{d}\tau\notag\\
&\qquad+C\int_{0}^{t}\tau\|\Lambda^{\alpha}u(\tau)\|_{L^{2}}^{2}(\|\Lambda^{\alpha}u(\tau)\|_{L^{2}}^{2}+\|u(\tau)\|_{L^{2}}^{2})\mathrm{d}\tau\notag\\
&\leq
C_{0}{t}+\int_{0}^{t}\tau\|\Lambda^{\beta}b(\tau)\|_{L^{2}}^{2}\|\Lambda^{\frac{3\beta}{2}}b(\tau)\|_{L^{2}}^{2}\mathrm{d}\tau\notag\\
&\leq C_{0}(t)+C_{0}\int_{0}^{t}\|\Lambda^{\frac{3\beta}{2}}b(\tau)\|_{L^{2}}^{2}\mathrm{d}\tau\notag\\
&\leq C_{0}(t),
\end{align*}
where we have used the a priori estimate \eqref{3.2} and the inequalities \eqref{3.22}- \eqref{3.23}.
Similarly, we can get from  the equation $\eqref{1.1}_{3}$ that
\begin{align*}
\|\Lambda^{2\beta}b\|_{L^{2}}&\leq\|\partial_{t}b\|_{L^{2}}+\|b\cdot\nabla
u\|_{L^{2}}+\|u\cdot\nabla b\|_{L^{2}}\notag\\
&\leq
C\|\partial_{t}b\|_{L^{2}}+C\|b\|_{L^{\frac{2n}{n-2\beta}}}\|\nabla
u\|_{L^{\frac{n}{\beta}}}+C\|u\|_{L^{\frac{2n}{n-2\alpha}}}\|\nabla
b\|_{L^{\frac{n}{\alpha}}}\notag\\
&\leq
C\|\partial_{t}b\|_{L^{2}}+C\|\Lambda^{\beta}b\|_{L^{2}}\|\Lambda^{\alpha}
u\|_{L^{2}}+C\|\Lambda^{\alpha}u\|_{L^{2}}\|\Lambda^{\beta}
b\|_{L^{2}}\notag\\
&\leq
C\|\partial_{t}b\|_{L^{2}}+C\|\Lambda^{\beta}b\|_{L^{2}}^{2}+C\|\Lambda^{\alpha}
u\|_{L^{2}}^{2},
\end{align*}
then we get $$\int_{0}^{t}\tau\|\Lambda^{2\beta}b(\tau)\|_{L^{2}}^{2}\mathrm{d}\tau\leq
C_{0}(t).$$
Clearly, we also have
$$\int_{0}^{t}\|(\Lambda^{2\alpha}u,\Lambda^{2\beta}b)(\tau)\|_{L^{2}}^{2}\leq
C_{0}.$$
We thus complete the proof of Lemma \ref{lemma3.2}.
\end{proof}
The following Lemma \ref{lemma3.3} is used for the estimation of high
order derivative estimates for $(u,b)$ .


\begin{lemma}\label{lemma3.3}
Under the assumptions of Theorem \ref{thm1}, the corresponding
solution $(\rho,u,b)$ of the equations \eqref{1.1} admits the
following a priori bounds for any $t>0$
\begin{align}\label{3.25}
t^{k}\|(\sqrt{\rho}\partial_{t}u,\partial_{t}b)(t)\|_{L^{2}}^{2}+\int_{0}^{t}\tau^{k}\|(\Lambda^{\alpha}\partial_{\tau}u,\Lambda^{\alpha}\partial_{\tau}b)(\tau)\|_{L^{2}}^{2}\leq
C_{0}(t^{k}),k=0,1,
\end{align}
\begin{align}\label{3.26}
\int_{0}^{t}\|\nabla P\|_{L^{2}}^{2}\mathrm{d}\tau\leq C_{0},
\end{align}
where $C_{0}$ depends only on
$\|\rho_0\|_{L^{\infty}}$, $\|\sqrt{\rho_{0}}u_{0}\|_{L^{2}}$, $\|\Lambda^{\alpha}u_{0}\|_{L^{2}}$ and $\|b_{0}\|_{H^{\beta}}$.
\end{lemma}


\begin{proof}
Differentiating the equation $\eqref{1.1}_{2}$ with respect to $t$ gives
\begin{align}\label{3.27}
\rho\partial_{tt}u+\rho
u\cdot\nabla\partial_{t}u+(-\Delta)^{\alpha}\partial_{t}u+\nabla\partial_{t}P=-\partial_{t}\rho\partial_{t}u-\partial_{t}(\rho
u)\cdot\nabla u+\partial_{t}(b\cdot\nabla b).
\end{align}
Multiplying $\partial_{t}u$ on the both sides of \eqref{3.27}  and integrating  by parts, we obtain after using the equation
$\eqref{1.1}_{1}$ that
\begin{align}\label{3.28}
\frac{\mathrm{d}}{\mathrm{d}t}\|\sqrt{\rho}\partial_{t}u\|_{L^{2}}^{2}+\|\Lambda^{\alpha}\partial_{t}u\|_{L^{2}}^{2}
&=-\int_{\mathbb{R}^{n}}\partial_{t}\rho\cdot\partial_{t}u\cdot\partial_{t}u\mathrm{d}x
-\int_{\mathbb{R}^{n}}\partial_{t}\rho u\cdot\nabla
u\cdot\partial_{t}u\mathrm{d}x\notag\\
&\qquad-\int_{\mathbb{R}^{n}}\rho\partial_{t}u\cdot\nabla
u\cdot\partial_{t}u\mathrm{d}x+\int_{\mathbb{R}^{n}}\partial_{t}(b\cdot\nabla
b)\cdot\partial_{t}u\mathrm{d}x\notag\\
&=-2\int_{\mathbb{R}^{n}}\rho
u\cdot\nabla\partial_{t}u\cdot\partial_{t}u\mathrm{d}x
-\int_{\mathbb{R}^{n}}\rho u\nabla( u\cdot\nabla
u\cdot\partial_{t}u)\mathrm{d}x\notag\\
&\qquad-\int_{\mathbb{R}^{n}}\rho\partial_{t}u\cdot\nabla
u\cdot\partial_{t}u\mathrm{d}x+\int_{\mathbb{R}^{n}}\partial_{t}(b\cdot\nabla
b)\cdot\partial_{t}u\mathrm{d}x\notag\\
&\triangleq\sum_{i=1}^{4}J_{i}.
\end{align}
Next, we estimate the terms on the right hand side one by one. Using  H\"older, Gagliardo-Nirenberg and
Young's inequalities, we get
\begin{align*}
\mid
J_{1}\mid&\leq C\|\sqrt{\rho}\|_{L^{\infty}}\|\sqrt{\rho}\partial_{t}u\|_{L^{2}}\|\nabla\partial_{t}u\|_{L^{\frac{n}{\beta}}}\|u\|_{L^{\frac{2n}{n-2\beta}}}\notag\\
&\leq C\|\rho_{0}\|_{L^{\infty}}^{\frac{1}{2}}\|\sqrt{\rho}\partial_{t}u\|_{L^{2}}\|\Lambda^{\alpha}\partial_{t}u\|_{L^{2}}\|\Lambda^{\beta}u\|_{L^{2}}\notag\\
&\leq C\|\sqrt{\rho}\partial_{t}u\|_{L^{2}}\|\Lambda^{\alpha}\partial_{t}u\|_{L^{2}}\|u\|_{L^{2}}^{1-\frac{\beta}{\alpha}}\|\Lambda^{\alpha}u\|_{L^{2}}^{\frac{\beta}{\alpha}}\notag\\
&\leq C\|\sqrt{\rho}\partial_{t}u\|_{L^{2}}\|\Lambda^{\alpha}\partial_{t}u\|_{L^{2}}(\|u\|_{L^{2}}+\|\Lambda^{\alpha}u\|_{L^{2}})\notag\\
&\leq
\frac{1}{16}\|\Lambda^{\alpha}\partial_{t}u\|_{L^{2}}^{2}+C\|\sqrt{\rho}\partial_{t}u\|_{L^{2}}^{2}(\|u\|_{L^{2}}^{2}+\|\Lambda^{\alpha}u\|_{L^{2}}^{2}).
\end{align*}
\begin{align*}
\mid J_{2}\mid&\leq \mid\int_{\mathbb{R}^{n}}\rho u\cdot\nabla
u\cdot\nabla
u\cdot\partial_{t}u\mathrm{d}x\mid+\mid\int_{\mathbb{R}^{n}}\rho
u\cdot u\cdot\nabla^{2}
u\cdot\partial_{t}u\mathrm{d}x\mid+\mid\int_{\mathbb{R}^{n}}\rho
u\cdot u\cdot\nabla u\cdot\nabla\partial_{t}u\mathrm{d}x\mid\\
&\leq
C\|\rho\|_{L^{\infty}}\|u\|_{L^{\frac{2n}{n-2\alpha}}}\|\nabla
u\|_{L^{\frac{n}{\alpha}}}^{2}\|\partial_{t}u\|_{L^{\frac{2n}{n-2\alpha}}}+C\|\sqrt{\rho}\|_{L^{\infty}}\|\sqrt{\rho}\partial_{t}u\|_{L^{2}}\|
u\|_{L^{\frac{n}{\alpha-1}}}^{2}\|\nabla^{2}u\|_{L^{\frac{2n}{4+n-4\alpha}}}\\&\qquad+C\|\rho\|_{L^{\infty}}\|u\|_{L^{\frac{2n}{n-2\beta}}}^{2}\|\nabla
u\|_{L^{\frac{n}{\beta}}}\|\nabla\partial_{t}u\|_{L^{\frac{n}{\beta}}}\\
&\leq
C\|\rho_{0}\|_{L^{\infty}}\|\Lambda^{\alpha}u\|_{L^{2}}\|\Lambda^{\beta}
u\|_{L^{2}}^{2}\|\Lambda^{\alpha}\partial_{t}u\|_{L^{2}}+C\|\rho_{0}\|_{L^{\infty}}^{\frac{1}{2}}\|\sqrt{\rho}\partial_{t}u\|_{L^{2}}\|\Lambda^{\beta}
u\|_{L^{2}}^{2}\|\Lambda^{2\alpha}u\|_{L^{2}}\\&\qquad+C\|\rho_{0}\|_{L^{\infty}}\|\Lambda^{\beta}u\|_{L^{2}}^{2}\|\Lambda^{\alpha}
u\|_{L^{2}}\|\Lambda^{\alpha}\partial_{t}u\|_{L^{2}}\\
&\leq
C\|\Lambda^{\alpha}u\|_{L^{2}}(\|\Lambda^{\alpha}u\|_{L^{2}}^{2}+\|u\|_{L^{2}}^{2})\|\Lambda^{\alpha}\partial_{t}u\|_{L^{2}}+C\|\sqrt{\rho}\partial_{t}u\|_{L^{2}}(\|\Lambda^{\alpha}
u\|_{L^{2}}^{2}+\|u\|_{L^{2}}^{2})\|\Lambda^{2\alpha}u\|_{L^{2}}\\
&\leq\frac{1}{16}\|\Lambda^{\alpha}\partial_{t}u\|_{L^{2}}^{2}+C\|(u,\Lambda^{\alpha}u)\|_{L^{2}}^{4}\|(\Lambda^{\alpha}u,\Lambda^{\beta}b)\|_{L^{2}}^{2}+C\|\sqrt{\rho}\partial_{t}u\|_{L^{2}}^{2}\|(u,b,\Lambda^{\alpha}u,\Lambda^{\frac{3}{2}\beta}b)\|_{L^{2}}^{2},
\end{align*}
where we have applied the following estimate
$$\|\Lambda^{2\alpha}u\|_{L^{2}}\leq\|\Lambda^{\beta}b\|_{L^{2}}(\|b\|_{L^{2}}+\|\Lambda^{\frac{3\beta}{2}}b\|_{L^{2}})+\|\sqrt{\rho}\partial_{t}u\|_{L^{2}}+\|\Lambda^{\alpha}u\|_{L^{2}}(\|\Lambda^{\alpha}u\|_{L^{2}}+\|u\|_{L^{2}}),$$
which is proved in the estimate \eqref{u2alpha}.

The term $J_3$ can be similarly estimated as follows
\begin{align*}
\mid J_{3}\mid&\leq
C\|\sqrt{\rho}\|_{L^{\infty}}\|\sqrt{\rho}\partial_{t}u\|_{L^{2}}\|\nabla
u\|_{L^{\frac{n}{\alpha}}}\|\partial_{t}u\|_{L^{\frac{2n}{n-2\alpha}}}\\
&\leq
C\|\rho_{0}\|_{L^{\infty}}^{\frac{1}{2}}\|\sqrt{\rho}\partial_{t}u\|_{L^{2}}\|\Lambda^{\alpha}\partial_{t}u\|_{L^{2}}\|\Lambda^{\beta}u\|_{L^{2}}\\
&\leq
C\|\sqrt{\rho}\partial_{t}u\|_{L^{2}}\|\Lambda^{\alpha}\partial_{t}u\|_{L^{2}}\|u\|_{L^{2}}^{1-\frac{\beta}{\alpha}}\|\Lambda^{\alpha}u\|_{L^{2}}^{\frac{\beta}{\alpha}}\\
&\leq
\frac{1}{16}\|\Lambda^{\alpha}\partial_{t}u\|_{L^{2}}^{2}+C\|\sqrt{\rho}\partial_{t}u\|_{L^{2}}^{2}(\|u\|_{L^{2}}^{2}+\|\Lambda^{\alpha}u\|_{L^{2}}^{2}).
\end{align*}
For the last term $J_4$ we can get that by integration by parts and a similar argument
 \begin{align*} \mid
J_{4}\mid&=\mid\int_{\mathbb{R}^{n}}\partial_{t}(b_{i}\partial_{i}b_{j})\partial_{t}u_{j}\mathbb{d}x\mid\\
&=\mid-\int_{\mathbb{R}^{n}}\partial_{t}(b_{i}b_{j})\partial_{t}\partial_{i}u_{j}\mathbb{d}x\mid\\
&\leq
C\|\partial_{t}b\|_{L^{2}}\|b\|_{L^{\frac{2n}{n-2\beta}}}\|\partial_{t}\nabla
u\|_{L^{\frac{n}{\beta}}}\\
&\leq
C\|\partial_{t}b\|_{L^{2}}\|\Lambda^{\beta}b\|_{L^{2}}\|\Lambda^{\alpha}\partial_{t}u\|_{L^{2}}\\
&\leq\frac{1}{16}\|\Lambda^{\alpha}\partial_{t}u\|_{L^{2}}^{2}+C\|\partial_{t}b\|_{L^{2}}^{2}\|\Lambda^{\beta}b\|_{L^{2}}^{2}.
\end{align*}
Inserting all the above estimates into the equality \eqref{3.28} it follows that
\begin{align}\label{3.29}
\frac{\mathrm{d}}{\mathrm{d}t}\|\sqrt{\rho}\partial_{t}u\|_{L^{2}}^{2}+\|\Lambda^{\alpha}\partial_{t}u\|_{L^{2}}^{2}\leq&\|\sqrt{\rho}\partial_{t}u\|_{L^{2}}^{2}\|(u,b,\Lambda^{\alpha}u,\Lambda^{\frac{3}{2}\beta}b)\|_{L^{2}}^{2}+C\|\partial_{t}b\|_{L^{2}}^{2}\|\Lambda^{\beta}b\|_{L^{2}}^{2}
\notag\\&+C\|(u,\Lambda^{\alpha}u)\|_{L^{2}}^{4}\|(\Lambda^{\alpha}u,\Lambda^{\beta}b)\|_{L^{2}}^{2}.
\end{align}
Moreover, differentiating the equation $\eqref{1.1}_{3}$ with respect to $t$
shows
$$\partial_{tt}b+(u\cdot\nabla)\partial_{t}b+(-\Delta)^{\beta}\partial_{t}b=\partial_{t}(b\cdot\nabla
u)-\partial_{t}u\cdot\nabla b.$$
 Multiplying both sides of the above equation by $\partial_{t}b$
and integrating the resulting equality by parts and using a similar argument as deducing the estimate (\ref{3.29}) we obtain
\begin{align*}
\frac{1}{2}\frac{\mathrm{d}}{\mathrm{d}t}\|\partial_{t}b\|_{L^{2}}^{2}+\|\Lambda^{\beta}\partial_{t}b\|_{L^{2}}^{2}&=\int_{\mathbb{R}^{n}}\partial_{t}(b\cdot\nabla
u)\cdot\partial_{t}b\mathrm{d}x-\int_{\mathbb{R}^{n}}\partial_{t}u\cdot\nabla
b\cdot\partial_{t}b\mathrm{d}x\\
&=\int_{\mathbb{R}^{n}}\partial_{t}b\cdot\nabla
u\cdot\partial_{t}b\mathrm{d}x+\int_{\mathbb{R}^{n}}b\cdot\nabla\partial_{t}
u\cdot\partial_{t}b\mathrm{d}x-\int_{\mathbb{R}^{n}}\partial_{t}u\cdot\nabla
b\cdot\partial_{t}b\mathrm{d}x\\
&\leq\|\partial_{t}b\|_{L^{2}}\|\nabla
u\|_{L^{\frac{n}{\beta}}}\|\partial_{t}b\|_{L^{\frac{2n}{n-2\beta}}}+\|b\|_{L^{\frac{2n}{n-2\beta}}}\|\nabla\partial_{t}u\|_{L^{\frac{n}{\beta}}}\|\partial_{t}b\|_{L^{2}}\\
&\qquad+\|\partial_{t}b\|_{L^{2}}\|\partial_{t}u\|_{L^{\frac{2n}{n-2\alpha}}}\|\nabla
b\|_{L^{\frac{n}{\alpha}}}\\
&\leq\frac{1}{16}\|\Lambda^{\alpha}\partial_{t}u\|_{L^{2}}^{2}+\frac{1}{2}\|\Lambda^{\beta}\partial_{t}b\|_{L^{2}}^{2}+\|\partial_{t}b\|_{L^{2}}^{2}(\|\Lambda^{\alpha}u\|_{L^{2}}^{2}+\|\Lambda^{\beta}b\|_{L^{2}}^{2}).
\end{align*}
Summing up the above estimate with the estimate \eqref{3.29} and using $c_{0}\|u\|_{L^{2}}\leq\|\sqrt{\rho}u\|_{L^{2}}$, we get
\begin{align}\label{3.30}
&\quad\frac{\mathrm{d}}{\mathrm{d}t}\|(\sqrt{\rho}\partial_{t}u,\partial_{t}b)\|_{L^{2}}^{2}+\|(\Lambda^{\alpha}\partial_{t}u,\Lambda^{\beta}\partial_{t}b)\|_{L^{2}}^{2}\notag\\
&\leq
C\|(\sqrt{\rho}\partial_{t}u,\partial_{t}b)\|_{L^{2}}^{2}\|(\sqrt{\rho}u,b,\Lambda^{\alpha}u,\Lambda^{\frac{3}{2}\beta}b)\|_{L^{2}}^{2}+\|(\sqrt{\rho}u,\Lambda^{\alpha}u)\|_{L^{2}}^{4}\|(\Lambda^{\alpha}u,\Lambda^{\beta}b)\|_{L^{2}}^{2}.
\end{align}
Using the Gr\"{o}nwall Lemma \ref{lemma1} and the a priori estimates \eqref{3.2}, \eqref{3.4}, it
yields
$$\|(\sqrt{\rho}\partial_{t}u,\partial_{t}b)(t)\|_{L^{2}}^{2}+\int_{0}^{t}\|(\Lambda^{\alpha}\partial_{\tau}u,\Lambda^{\beta}\partial_{\tau}b)(\tau)\|_{L^{2}}^{2}\leq
C_{0}(t).$$
In addition, we can get from the estimate \eqref{3.30} by multiplying time $t$ that
\begin{align*}
&\quad\frac{\mathrm{d}}{\mathrm{d}t}(t\|(\sqrt{\rho}\partial_{t}u,\partial_{t}b)\|_{L^{2}}^{2})+t\|(\Lambda^{\alpha}\partial_{t}u,\Lambda^{\beta}\partial_{t}b)\|_{L^{2}}^{2}\notag\\
&\leq\|(\sqrt{\rho}\partial_{t}u,\partial_{t}b)\|_{L^{2}}^{2}+
Ct\|(\sqrt{\rho}\partial_{t}u,\partial_{t}b)\|_{L^{2}}^{2}\|(\sqrt{\rho}u,b,\Lambda^{\alpha}u,\Lambda^{\frac{3}{2}\beta}b)\|_{L^{2}}^{2}\\&\qquad+\|(\sqrt{\rho}u,\Lambda^{\alpha}u)\|_{L^{2}}^{4}t\|(\Lambda^{\alpha}u,\Lambda^{\beta}b)\|_{L^{2}}^{2}.
\end{align*}
By  Lemma \ref{3.2} and  the Gr\"{o}nwall Lemma \ref{lemma1}, we have
\begin{align*}
t\|\sqrt{\rho}\partial_{t}u\|_{L^{2}}^{2}+t\|\partial_{t}b\|_{L^{2}}^{2}+\int_{0}^{t}\tau\|\Lambda^{\alpha}\partial_{\tau}u\|_{L^{2}}^{2}+\tau\|\Lambda^{\beta}\partial_{\tau}b\|_{L^{2}}^{2}\mathrm{d}\tau
\leq C_{0}.
\end{align*}
It follows from the Stokes equations (\ref{3.24}) that
\begin{align*}
\|\nabla P\|_{L^{2}}&\leq C\|\rho\partial_{t}u\|_{L^{2}}+C\|\rho
u\cdot\nabla u\|_{L^{2}}+C\|b\cdot\nabla b\|_{L^{2}}\\
&\leq
C\|\sqrt{\rho}\partial_{t}u\|_{L^{2}}+C\|\Lambda^{\alpha}u\|_{L^{2}}(\|\Lambda^{\alpha}u\|_{L^{2}}+\|u\|_{L^{2}})+C\|\Lambda^{\beta}b\|_{L^{2}}(\|\Lambda^{\frac{3}{2}\beta}b\|_{L^{2}}+\|b\|_{L^{2}}).
\end{align*}
We thus deduce by the a priori estimates \eqref{3.2}, \eqref{3.4} and the inequality \eqref{3.19} that
$$\int_{0}^{t}\|\nabla P\|_{L^{2}}^{2}\mathrm{d}\tau\leq C_{0}.$$
We thus conclude the proof of Lemma \ref{3.3}.
\end{proof}
The following Lemma \ref{3.4}  plays a key role in proving the
uniqueness of a solution.


\begin{lemma}\label{lemma3.4}
Under the assumptions of Theorem \ref{thm1}, the corresponding
solution $(\rho,u,b)$ of the equations \eqref{1.1} admits the
following bounds for any $t>0$
\begin{align}\label{3.31}
\int_{0}^{t}\|\nabla u(\tau)\|_{L^{\infty}}\mathrm{d}\tau\leq
C_{0}(t),
\end{align}
\begin{align}\label{3.32}
\|\nabla\rho(t)\|_{L^{\frac{2n}{6\alpha-n}}}\leq C_{0}(t),
\end{align}
where $C_{0}(t)$ depends only on the initial data and the time $t$.
\end{lemma}


\begin{proof}
First, for any $2<p<\frac{2n}{n-2\alpha}$, we have by the interpolation inequality of Lebesgue spaces
\begin{align}\label{3.34}
\|\rho\partial_{t}u\|_{L^{p}}&\leq\|\rho\partial_{t}u\|_{L^{2}}^{\theta}\|\rho\partial_{t}u\|_{L^{\frac{2n}{n-2\alpha}}}^{1-\theta}\notag\\
&\leq\|\sqrt{\rho}\|_{L^{\infty}}^{\theta}\|\sqrt{\rho}\partial_{t}u\|_{L^{2}}^{\theta}\|\rho\|_{L^{\infty}}^{1-\theta}\|\partial_{t}u\|_{L^{\frac{2n}{n-2\alpha}}}^{1-\theta}\notag\\
&\leq\|\sqrt{\rho}\partial_{t}u\|_{L^{2}}^{\theta}\|\Lambda^{\alpha}\partial_{t}u\|_{L^{2}}^{1-\theta},
\end{align}
where
$\theta\triangleq1-\frac{n}{2\alpha}+\frac{n}{p\alpha}\in(0,1)$.
Next, by the H\"older and Sobolev's inequalities we have
\begin{align}\label{3.35}
\|u\cdot\nabla u\|_{L^{\frac{2n}{n-2\alpha}}}&\leq
C\|u\|_{L^{\infty}}\|\nabla u\|_{L^{\frac{2n}{n-2\alpha}}}\notag\\
&\leq
C\|\Lambda^{\alpha}u\|_{L^{2}}^{2-\frac{n}{2\alpha}}\|\Lambda^{2\alpha}u\|_{L^{2}}^{\frac{n}{2\alpha}-1}\|\Lambda^{\alpha}u\|_{L^{2}}^{1-\frac{1}{\alpha}}\|\Lambda^{2\alpha}u\|_{L^{2}}^{\frac{1}{\alpha}}\notag\\
&\leq
C\|\Lambda^{\alpha}u\|_{L^{2}}^{3-\frac{n}{2\alpha}-\frac{1}{\alpha}}\|\Lambda^{2\alpha}u\|_{L^{2}}^{\frac{n}{2\alpha}+\frac{1}{\alpha}-1}
\end{align}
Thus, by the interpolation in Lebesgue space, it can be deduced
\begin{align}\label{3.36}
\|\rho u\cdot\nabla u\|_{L^{p}}&\leq C\|\rho u\cdot\nabla
u\|_{L^{2}}^{\theta}\|\rho u\cdot\nabla
u\|_{L^{\frac{2n}{n-2\alpha}}}^{1-\theta}\notag\\
&\leq\|\rho\|_{L^{\infty}}^{\theta}\| u\cdot\nabla
u\|_{L^{2}}^{\theta}\|\rho\|_{L^{\infty}}^{1-\theta}\|
u\cdot\nabla u\|_{L^{\frac{2n}{n-2\alpha}}}^{1-\theta}\notag\\
&\leq
C_{0}(\|\Lambda^{\alpha}u\|_{L^{2}}^{3-\frac{n}{2\alpha}-\frac{1}{\alpha}}\|\Lambda^{2\alpha}u\|_{L^{2}}^{\frac{n}{2\alpha}+\frac{1}{\alpha}-1})^{1-\theta}\notag\\
&\leq
C_{0}(\|\Lambda^{\alpha}u\|_{L^{2}}^{2}+\|\Lambda^{2\alpha}u\|_{L^{2}}^{2})^{1-\theta},
\end{align}
where we have used that
\begin{align*}
\| u\cdot\nabla u\|_{L^{2}}&\leq
C\|u\|_{L^{\frac{2n}{n-2\alpha}}}\|\nabla
u\|_{L^{\frac{n}{\alpha}}}\\
&\leq C\|\Lambda^{\alpha}u\|_{L^{2}}\|\Lambda^{\beta}u\|_{L^{2}}\\
&\leq
C\|\Lambda^{\alpha}u\|_{L^{2}}(\|\Lambda^{\alpha}u\|_{L^{2}}+\|u\|_{L^{2}})\\&\leq
C_{0}.
\end{align*}
Using a similar argument as \eqref{3.36}, it can be obtained that
\begin{align}\label{3.37}
 \|b\cdot\nabla
b\|_{L^{p}}\leq
C(\|\Lambda^{\beta}b\|_{L^{2}}^{2}+\|\Lambda^{2\beta}b\|_{L^{2}}^{2})^{1-\theta}.
\end{align}
Next, applying the $L^{p}$-estimate to Stoke equations \eqref{3.24}, we arrive at
\begin{align}\label{3.38}
\|\Lambda^{2\alpha}u\|_{L^{p}}\leq
C\|\rho\partial_{t}u\|_{L^{p}}+C\|\rho u\cdot\nabla
u\|_{L^{p}}+C\|b\cdot\nabla b\|_{L^{p}}.
\end{align}
Now, by the Gagliardo-Nirenberg's inequality, inserting \eqref{3.34}, \eqref{3.36}-\eqref{3.38}, it can be deduced that for $p>\frac{n}{2\alpha-1}$
\begin{align*}
\|\nabla u\|_{L^{\infty}}&\leq C\|\nabla
u\|_{L^{\frac{n}{\beta}}}^{1-\eta}\|\Lambda^{2\alpha}u\|_{L^{p}}^{\eta}\leq
C\|\Lambda^{\alpha}u\|_{L^{2}}^{1-\eta}\|\Lambda^{2\alpha}u\|_{L^{p}}^{\eta}\\
&\leq
C\|\Lambda^{\alpha}u\|_{L^{2}}^{1-\eta}(\|\rho\partial_{t}u\|_{L^{p}}+(\|\Lambda^{\alpha}u\|_{L^{2}}^{2}+\|\Lambda^{2\alpha}u\|_{L^{2}}^{2}+\|\Lambda^{\beta}b\|_{L^{2}}^{2}+\|\Lambda^{2\beta}b\|_{L^{2}}^{2})^{1-\theta})^{\eta},
\end{align*}
where $\eta=\frac{(2+n-2\alpha)p}{p(n+2\alpha)-2n}\in(0,1)$.
Using
Lemma \ref{3.1} and Lemma \ref{3.2}, we
have
\begin{align*}
\int_{0}^{t}\|\nabla u(\tau)\|_{L^{\infty}}\mathrm{d}\tau\leq
C_{0}(t).
\end{align*}
Finally, applying $\nabla$ on the both sides of the equation $\eqref{1.1}_{1}$, we get
$$\partial_{t}\nabla\rho+u\cdot\nabla(\nabla\rho)=-\nabla
u\cdot\nabla\rho.$$
 Due to the divergence free condition $\nabla\cdot u=0,$ we have
$$\frac{\mathrm{d}}{\mathrm{d}t}\|\nabla\rho\|_{L^{\frac{2n}{6\alpha-n}}}\leq\|\nabla
u\|_{L^{\infty}}\|\nabla\rho\|_{L^{\frac{2n}{6\alpha-n}}}.$$
 Using
the Gronwall Lemma \ref{lemma1} and the a priori estimate \eqref{3.31} it follows that
\begin{align*}
\|\nabla\rho\|_{L^{\frac{2n}{6\alpha-n}}}\leq\|\nabla\rho_{0}\|_{L^{\frac{2n}{6\alpha-n}}}\exp\int_{0}^{t}\|\nabla
u(\tau)\|_{L^{\infty}}\mathrm{d}\tau\leq C_{0}(t).
\end{align*}
We thus complete the proof of Lemma \ref{3.4}.
\end{proof}


\section{Proof of Theorem \ref{thm1}}
\setcounter{equation}{0}

 This section aims at the proof of
Theorem \ref{thm1}. Thanks to Lemmas \ref{3.1}-\ref{3.4} and by the local existence
Proposition \ref{proposition1}, we can get the global existence of
strong solutions directly. In the following part, we prove the
continuity and uniqueness of the solution which we have just constructed.


\begin{proof}
Firstly, we prove the time continuity of the solution, namely
\begin{align}\label{4.1} \rho\in
C([0,T];L^{p}(\mathbb{R}^{n})),\frac{n}{2\alpha}\leq p<\infty,
\end{align}
\begin{align}\label{4.2}
\rho u\in C([0,T];L^{2}(\mathbb{R}^{n})).
\end{align}
\begin{align}\label{4.3}
 u\in C([0,T];\dot
H^{\alpha}(\mathbb{R}^{n})).
\end{align}
\begin{align}\label{4.4}
 b\in C([0,T];
H^{\beta}(\mathbb{R}^{n})).
\end{align}
To prove the
continuity on $[0,T]$, without loss of generality,  we only need to prove the continuity  at the initial time $t=0$. Since
$\partial_{t}\rho=-u\cdot\nabla\rho,$ it has
\begin{align*}
\|\rho(t)-\rho(0)\|_{L^{\frac{n}{2\alpha}}}&=\|\int_{0}^{t}\partial_{t}\rho(\tau)\mathrm{d}\tau\|_{L^{\frac{n}{2\alpha}}}\\
&=\|\int_{0}^{t}u\cdot\nabla\rho(\tau)\mathrm{d}\tau\|_{L^{\frac{n}{2\alpha}}}\\
&\leq\int_{0}^{t}\|u\cdot\nabla\rho(\tau)\|_{L^{\frac{n}{2\alpha}}}\mathrm{d}\tau\\
&\leq\int_{0}^{t}\|u(\tau)\|_{L^{\frac{2n}{n-2\alpha}}}\|\nabla\rho(\tau)\|_{L^{\frac{2n}{6\alpha-n}}}\mathrm{d}\tau\\
&\leq\int_{0}^{t}\|\Lambda^{\alpha}u(\tau)\|_{L^{2}}\|\nabla\rho(\tau)\|_{L^{\frac{2n}{6\alpha-n}}}\mathrm{d}\tau\\
&\leq C_{0}(t)t,
\end{align*}
where in the last line we have used the a priori estimates \eqref{3.4} and \eqref{3.32}.
By the H\"older inequality, one has
$$\|\rho(t)-\rho(0)\|_{L^{p}}\leq C\|\rho(t)-\rho(0)\|_{L^{\frac{n}{2\alpha}}}^{\frac{n}{2\alpha
p}}\|\rho(t)-\rho(0)\|_{L^{\infty}}^{1-\frac{n}{2\alpha p}}\leq
C_{0}(t)t^{\frac{n}{2\alpha p}},$$
which implies that $\rho$ is
continuous at the original time and satisfies
$\lim_{t\rightarrow 0}\|\rho-\rho_{0}\|_{L^{p}}=0$.  To
prove \eqref{4.2}, we first notice that
\begin{align*}
\|\rho
u(t)-\rho u(0)\|_{L^{\frac{2n}{n+2\alpha}}}&=\|\int_{0}^{t}
\partial_{t}(\rho
u)(\tau)\mathrm{d}\tau\|_{L^{\frac{2n}{n+2\alpha}}}\\
&=\|\int_{0}^{t}
\partial_{t}\rho
u(\tau)+\rho\partial_{t}u(\tau)\mathrm{d}\tau\|_{L^{\frac{2n}{n+2\alpha}}}\\
&\leq\int_{0}^{t}\|
\partial_{t}\rho
u(\tau)\|_{L^{\frac{2n}{n+2\alpha}}}\mathrm{d}\tau+\int_{0}^{t}\|\rho\partial_{t}u(\tau)\|_{L^{\frac{2n}{n+2\alpha}}}\mathrm{d}\tau\\
&\leq\int_{0}^{t}\| u\cdot\nabla\rho
u(\tau)\|_{L^{\frac{2n}{n+2\alpha}}}\mathrm{d}\tau+\int_{0}^{t}\|\sqrt{\rho}\sqrt{\rho}\partial_{t}u(\tau)\|_{L^{\frac{2n}{n+2\alpha}}}\mathrm{d}\tau\\
&\leq\int_{0}^{t}\|\nabla\rho
(\tau)\|_{L^{\frac{2n}{6\alpha-n}}}\|u(\tau)\|_{L^{\frac{2n}{n-2\alpha}}}^{2}\mathrm{d}\tau+\int_{0}^{t}\|\rho(\tau)\|_{L^{\frac{n}{2\alpha}}}^{\frac{1}{2}}\|\sqrt{\rho}\partial_{t}u(\tau)\|_{L^{2}}\mathrm{d}\tau\\
&\leq C_{0}(t)\int_{0}^{t}\|\Lambda^{\alpha}u(\tau)\|_{L^{2}}^{2}\mathrm{d}\tau+C_{0}\int_{0}^{t}\|\sqrt{\rho}\partial_{t}u(\tau)\|_{L^{2}}\mathrm{d}\tau\\
&\leq C_{0}(t)t+C_{0}t^{\frac{1}{2}}(\int_{0}^{t}\|\sqrt{\rho}\partial_{t}u(\tau)\|_{L^{2}}^{2}\mathrm{d}\tau)^{\frac{1}{2}}\\
&\leq C_{0}(t)t+C_{0}t^{\frac{1}{2}},
\end{align*}
 By the
H\"older inequality, one has
\begin{align*}
\|\rho u(t)-\rho u(0)\|_{L^{2}}&\leq\|\rho
u(t)-\rho_{0}u_{0}\|_{L^{\frac{2n}{n+2\alpha}}}^{\frac{1}{2}}\|\rho
u(t)-\rho_{0}u_{0}\|_{L^{\frac{2n}{n-2\alpha}}}^{\frac{1}{2}}\\
&\leq(C_{0}(t)t+C_{0}t^{\frac{1}{2}})^{\frac{1}{2}}.
\end{align*}
Thus we have that $\rho u$ is continuous at the original time and
satisfies $\lim_{t\rightarrow 0}\|\rho u-\rho u(0)\|_{L^{2}}=0$.  To prove \eqref{4.3}, we can obtain by direct
calculation,
\begin{align*}
\|\Lambda^{\alpha}u(t)-\Lambda^{\alpha}u(0)\|_{L^{2}}&=\|\int_{0}^{t}\partial_{\tau}(\Lambda^{\alpha}u)(\tau)\mathrm{d}\tau\|_{L^{2}}\\
&\leq\int_{0}^{t}\|\partial_{\tau}(\Lambda^{\alpha}u)(\tau)\|_{L^{2}}\mathrm{d}\tau\\
&\leq C(\int_{0}^{t}\|\Lambda^{\alpha}\partial_{\tau}u(\tau)\|_{L^{2}}^{2}\mathrm{d}\tau)^{\frac{1}{2}}\cdot t^{\frac{1}{2}}\\
&\leq C_{0}t^{\frac{1}{2}},
\end{align*}
 Using a similar argument, we get
\begin{align*}
\|\Lambda^{\beta}b(t)-\Lambda^{\beta}b(0)\|_{L^{2}}\leq C_{0}t^{\frac{1}{2}}
\end{align*}
 and
\begin{align*}
\|b(t)-b(0)\|_{L^{2}}
\leq C_{0}t^{\frac{1}{2}}.
\end{align*}
We obtain that
$\|b\|_{H^{\beta}}$ is continuous at the original time and
satisfies $\|b\|_{H^{\beta}}|_{t=0}=\|b_{0}\|_{H^{\beta}}$.

 Next, we show the uniqueness of the
solution. Assume that $(\rho^{(1)},u^{(1)},b^{(1)})$ and
$(\rho^{(2)},u^{(2)},b^{(2)})$ are two solutions of the equations
\eqref{1.1} emanating from the same initial data $(\rho_0, u_0, b_0)$, and possess
the properties in Theorem \ref{1.1}. Their difference
$(\tilde{\rho},\tilde{u},\tilde{b})$ denoted by
$$\tilde{\rho}=\rho^{(2)}-\rho^{(1)},\tilde{u}=u^{(2)}-u^{(1)},\tilde{b}=b^{(2)}-b^{(1)}$$
satisfies the following equations
\begin{equation}\label{4.5}
\begin{cases}
\partial_{t}\tilde{\rho}+u^{(2)}\cdot\nabla\tilde{\rho}=-\tilde{u}\cdot\nabla\rho^{(1)},\\
\rho^{(2)}\partial_{t}\tilde{u}+\rho^{(2)}u^{(2)}\cdot\nabla\tilde{u}+(-\Delta)^{\alpha}\tilde{u}+\nabla (P^{(2)}-P^{(1)})\\
 =-\tilde{\rho}(\partial_{t}u^{(1)}+u^{(1)}\cdot\nabla u^{(1)})-\rho^{(2)}\tilde{u}\cdot\nabla u^{(1)}+\tilde{b}\cdot\nabla b^{(2)}-b^{(1)}\cdot\nabla\tilde{b}, \\
\partial_{t}\tilde{b}+u^{(2)}\cdot\nabla\tilde{b}+(-\Delta)^{\beta}\tilde{b}=\tilde{b}\cdot\nabla u^{(2)}+b^{(1)}\cdot\nabla\tilde{u}-\tilde{u}\cdot\nabla b^{(1)}.
\end{cases}
\end{equation}
We estimate the difference $(\tilde{\rho},\tilde{u},\tilde{b})$ in
$L^{2}(\mathbb{R}^{n})$ space. Dotting the equations $\eqref{4.5}_{2,3}$ by
$(\tilde{u},\tilde{b})$ and applying the divergence free
condition, we find
\begin{align}\label{4.6}
\frac{1}{2}\frac{\mathrm{d}}{\mathrm{d}t}(\|\sqrt{\rho^{(2)}}\tilde{u}(t)\|_{L^{2}}^{2}+\|\tilde{b}\|_{L^{2}}^{2})+\|\Lambda^{\alpha}\tilde{u}\|_{L^{2}}^{2}+\|\Lambda^{\beta}\tilde{b}\|_{L^{2}}^{2}=L_{1}+L_{2}+L_{3}+L_{4},
\end{align}
where
\begin{align*}
&L_{1}=-\int_{\mathbb{R}^{n}}\rho^{(2)}\tilde{u}\cdot\nabla
u^{(1)}\cdot\tilde{u}\mathrm{d}x,\qquad
L_{2}=-\int_{\mathbb{R}^{n}}\tilde{\rho}(\partial_{t}u^{(1)}+u^{(1)}\cdot\nabla u^{(1)})\cdot\tilde{u}\mathrm{d}x,\\
&L_{3}=\int_{\mathbb{R}^{n}}\tilde{b}\cdot\nabla
b^{(2)}\cdot\tilde{u}\mathrm{d}x,\qquad
L_{4}=\int_{\mathbb{R}^{n}}(\tilde{b}\cdot\nabla u^{(2)}-\tilde{u}\cdot\nabla b^{(1)})\cdot\tilde{b}\mathrm{d}x.
\end{align*}
By H\"older inequality, one has
 \begin{align}\label{4.7}
 L_{1}\leq C\|\nabla
u^{(1)}\|_{L^{\infty}}\|\sqrt{\rho^{(2)}}\tilde{u}\|_{L^{2}}^{2}.
\end{align}
By the H\"older, Sobolev and
Young's inequalities and the estimate \eqref{3.35}, it yields
\begin{align}\label{4.8}
L_{2}&\leq
C\|\tilde{\rho}\|_{L^{\frac{n}{2\alpha}}}\|\tilde{u}\|_{L^{\frac{2n}{n-2\alpha}}}(\|\partial_{t}u^{(1)}\|_{L^{\frac{2n}{n-2\alpha}}}+\|u^{(1)}\cdot\nabla
u^{(1)}\|_{L^{\frac{2n}{n-2\alpha}}})\notag\\
&\leq
C\|\tilde{\rho}\|_{L^{\frac{n}{2\alpha}}}\|\Lambda^{\alpha}\tilde{u}\|_{L^{2}}(\|\Lambda^{\alpha}\partial_{t}u^{(1)}\|_{L^{2}}+\|\Lambda^{\alpha}u\|_{L^{2}}^{3-\frac{n}{2\alpha}-\frac{1}{\alpha}}\|\Lambda^{2\alpha}u\|_{L^{2}}^{\frac{n}{2\alpha}+\frac{1}{\alpha}-1})\notag\\
&\leq\frac{1}{16}\|\Lambda^{\alpha}\tilde{u}\|_{L^{2}}^{2}+C\|\tilde{\rho}\|_{L^{\frac{n}{2\alpha}}}^{2}(\|\Lambda^{\alpha}\partial_{t}u^{(1)}\|_{L^{2}}^{2}
+\|\Lambda^{\alpha}u\|_{L^{2}}^{6-\frac{n}{\alpha}-\frac{2}{\alpha}}\|\Lambda^{2\alpha}u\|_{L^{2}}^{\frac{n}{\alpha}+\frac{2}{\alpha}-2}).
\end{align}
By a similar argument as deriving $L_{2}$, we obtain
\begin{align}\label{4.9}
L_{3}&\leq C\|\tilde{b}\|_{L^{2}}\|\nabla
b^{(2)}\|_{L^{\frac{n}{\alpha}}}\|\tilde{u}\|_{L^{\frac{2n}{n-2\alpha}}}\notag\\
&\leq
C\|\tilde{b}\|_{L^{2}}\|\Lambda^{\beta}b^{(2)}\|_{L^{2}}\|\Lambda^{\alpha}\tilde{u}\|_{L^{2}}\notag\\
&\leq
\frac{1}{16}\|\Lambda^{\alpha}\tilde{u}\|_{L^{2}}^{2}+C\|\Lambda^{\beta}b^{(2)}\|_{L^{2}}^{2}\|\tilde{b}\|_{L^{2}}^{2}
\end{align}
and
\begin{align}\label{4.10}
L_{4}&\leq C\|\nabla
u^{(2)}\|_{L^{\infty}}\|\tilde{b}\|_{L^{2}}^{2}+C\|\tilde{b}\|_{L^{2}}\|\nabla
b^{(1)}\|_{L^{\frac{n}{\alpha}}}\|\tilde{u}\|_{L^{\frac{2n}{n-2\alpha}}}\notag\\
&\leq C\|\nabla
u^{(2)}\|_{L^{\infty}}\|\tilde{b}\|_{L^{2}}^{2}+C\|\tilde{b}\|_{L^{2}}\|\Lambda^{\beta}
b^{(1)}\|_{L^{2}}\|\Lambda^{\alpha}\tilde{u}\|_{L^{2}}\notag\\
&\leq\frac{1}{16}\|\Lambda^{\alpha}\tilde{u}\|_{L^{2}}^{2}+\|\tilde{b}\|_{L^{2}}^{2}(\|\nabla
u^{(2)}\|_{L^{\infty}}+\|\Lambda^{\beta} b^{(1)}\|_{L^{2}}^{2}).
\end{align}
Inserting \eqref{4.7}-\eqref{4.10}  into \eqref{4.6}, it leads to
\begin{align*}
&\frac{\mathrm{d}}{\mathrm{d}t}(\|\sqrt{\rho^{2}}\tilde{u}(t)\|_{L^{2}}^{2}+\|\tilde{b}\|_{L^{2}}^{2})+\|\Lambda^{\alpha}\tilde{u}\|_{L^{2}}^{2}+\|\Lambda^{\beta}\tilde{b}\|_{L^{2}}^{2}\\
&\leq
C(\|\Lambda^{\alpha}\partial_{t}u^{(1)}\|_{L^{2}}^{2}+\|\Lambda^{\alpha}u\|_{L^{2}}^{6-\frac{n}{\alpha}-\frac{2}{\alpha}}\|\Lambda^{2\alpha}u\|_{L^{2}}^{\frac{n}{\alpha}+\frac{2}{\alpha}-2})\|\tilde{\rho}\|_{L^{\frac{n}{2\alpha}}}^{2}\\
&\qquad+C(\|\nabla u^{(2)}\|_{L^{\infty}}+\|\Lambda^{\beta}
b^{(1)}\|_{L^{2}}+\|\Lambda^{\beta}b^{(2)}\|_{L^{2}}^{2})(\|\tilde{b}\|_{L^{2}}^{2}+\|\sqrt{\rho^{(2)}}\tilde{u}\|_{L^{2}}^{2}).
\end{align*}
In order to close the above inequality, we need to derive the estimate of $\|\tilde{\rho}\|_{L^{\frac{n}{2\alpha}}}$ as follows:
\begin{align*}
\frac{\mathrm{d}}{\mathrm{d}t}\|\tilde{\rho}\|_{L^{\frac{n}{2\alpha}}}^{\frac{n}{2\alpha}}&\leq
C\|\tilde{\rho}\|_{L^{\frac{n}{2\alpha}}}^{\frac{n}{2\alpha}-1}\|\tilde{u}\cdot\nabla\rho^{(1)}\|_{L^{\frac{n}{2\alpha}}}\\
&\leq
C\|\tilde{\rho}\|_{L^{\frac{n}{2\alpha}}}^{\frac{n}{2\alpha}-1}\|\tilde{u}\|_{L^{\frac{2n}{n-2\alpha}}}\|\nabla\rho^{(1)}\|_{L^{\frac{2n}{6\alpha-n}}}\\
&\leq
C\|\tilde{\rho}\|_{L^{\frac{n}{2\alpha}}}^{\frac{n}{2\alpha}-1}\|\Lambda^{\alpha}\tilde{u}\|_{L^{2}}\|\nabla\rho^{(1)}\|_{L^{\frac{2n}{6\alpha-n}}},
\end{align*}
which leads to
$$\frac{\mathrm{d}}{\mathrm{d}t}\|\tilde{\rho}\|_{L^{\frac{n}{2\alpha}}}\leq
C\|\Lambda^{\alpha}\tilde{u}\|_{L^{2}}\|\nabla\rho^{(1)}\|_{L^{\frac{2n}{6\alpha-n}}}.$$
Now let us denote
\begin{align*}
&X(t)\triangleq\|\tilde{\rho}\|_{L^{\frac{n}{2\alpha}}},\qquad
Y(t)\triangleq\|\sqrt{\rho^{(2)}}\tilde{u}(t)\|_{L^{2}}^{2}+\|\tilde{b}(t)\|_{L^{2}}^{2},\\
&Z(t)\triangleq\|\Lambda^{\alpha}\tilde{u}\|_{L^{2}}^{2}+\|\Lambda^{\beta}\tilde{b}\|_{L^{2}}^{2},\qquad
A\triangleq\|\nabla\rho^{(1)}\|_{L^{\frac{2n}{6\alpha-n}}},\\
&\gamma(t)\triangleq C(\|\nabla u^{(2)}\|_{L^{\infty}}+\|\Lambda^{\beta}
b^{(1)}\|_{L^{2}}^{2}+\|\Lambda^{\beta}b^{(2)}\|_{L^{2}}^{2}),\\
&\eta(t)\triangleq
C(\|\Lambda^{\alpha}\partial_{t}u^{(1)}\|_{L^{2}}^{2}+\|\Lambda^{\alpha}u\|_{L^{2}}^{6-\frac{n}{\alpha}-\frac{2}{\alpha}}\|\Lambda^{2\alpha}u\|_{L^{2}}^{\frac{n}{\alpha}+\frac{2}{\alpha}-2}),
\end{align*}
which satisfy
\begin{align*}
\begin{cases}
\frac{\mathrm{d}}{\mathrm{d}t}X(t)\leq AZ^{\frac{1}{2}}(t),\\
\frac{\mathrm{d}}{\mathrm{d}t}Y(t)+Z(t)\leq\gamma(t)Y(t)+\eta(t)X^{2}(t), \\
X(0)=0, \\
\end{cases}
\end{align*}
According to the estimates of Lemmas \ref{3.1}-\ref{3.4}, we get
$$\int_{0}^{t}\gamma(\tau)\mathrm{d}\tau\leq C_{0}(t),\qquad
\int_{0}^{t}\tau\eta(\tau)\mathrm{d}\tau\leq C_{0}(t).$$ Thus, it
follows from Lemma \ref{2.2} that
$$\|\sqrt{\rho^{(2)}}\tilde{u}(t)\|_{L^{2}}^{2}+\|\tilde{b}(t)\|_{L^{2}}^{2}=\|(\Lambda^{\alpha}\tilde{u},\Lambda^{\alpha}\tilde{b})(\tau)\|_{L^{2}}^{2}=\|\tilde{\rho}\|_{L^{\frac{n}{2\alpha}}}=0.$$
which implies that
$$\rho^{(2)}\equiv \rho^{(1)},\qquad u^{(2)}\equiv u^{(1)},\qquad
b^{(2)}\equiv b^{(1)} \qquad \mbox{on}\ [0,T],$$
and we prove the uniqueness part of Theorem
\ref{1.1}. The proof of Theorem \ref{1.1} is thus completed.
\end{proof}


\textbf{Acknowledgements} The research of B Yuan was partially
supported by the National Natural Science Foundation of China (No.
11471103).

\end{document}